%% file: small.tex
\documentclass[12pt,oneside,ngerman,russian,english,reqno]{amsart}
\usepackage{mathpazo}
\usepackage[T1]{fontenc}
\usepackage[koi8-r,latin9]{inputenc}
\usepackage{geometry}
\usepackage{color}
\usepackage{babel}
\usepackage{units}
\usepackage{amsthm}
\usepackage{amstext}
\usepackage{amssymb}
\usepackage{esint}
\usepackage[unicode=true,pdfusetitle,
 bookmarks=true,bookmarksnumbered=false,bookmarksopen=false,
 breaklinks=false,pdfborder={0 0 0},backref=false,colorlinks=true]
 {hyperref}
\hypersetup{
 pdfborderstyle={},linkcolor=black,citecolor=black,urlcolor=blue}
\usepackage{breakurl}

\makeatletter

\AtBeginDocument{\DeclareFontEncoding{T2A}{}{}}

\theoremstyle{plain}
\newtheorem{thm}{\protect\theoremname}
 \theoremstyle{definition}
 \newtheorem*{defn*}{\protect\definitionname}
  \theoremstyle{plain}
  \newtheorem{lem}{\protect\lemmaname}
  \theoremstyle{remark}
  \newtheorem*{rem*}{\protect\remarkname}
  \theoremstyle{plain}
  \newtheorem*{conjecture*}{\protect\conjecturename}

\usepackage{calrsfs}
\usepackage{graphics}
\usepackage{color}
\usepackage{perpage}
\MakePerPage{footnote}
\gdef\SetFigFontNFSS#1#2#3#4#5{} 

\DeclareMathOperator{\PLA}{PLA}
\DeclareMathOperator{\conv}{conv}

\newcommand{\eps}{\varepsilon}

\theoremstyle{plain}
\newtheorem*{thm3p}{Theorem \ref{thm:notMS}'}

\theoremstyle{remark}
\newtheorem*{qst*}{Question}

\makeatother

  \addto\captionsenglish{\renewcommand{\conjecturename}{Conjecture}}
  \addto\captionsenglish{\renewcommand{\definitionname}{Definition}}
  \addto\captionsenglish{\renewcommand{\lemmaname}{Lemma}}
  \addto\captionsenglish{\renewcommand{\remarkname}{Remark}}
  \addto\captionsngerman{\renewcommand{\conjecturename}{Vermutung}}
  \addto\captionsngerman{\renewcommand{\definitionname}{Definition}}
  \addto\captionsngerman{\renewcommand{\lemmaname}{Lemma}}
  \addto\captionsngerman{\renewcommand{\remarkname}{Bemerkung}}
  \addto\captionsrussian{\renewcommand{\conjecturename}{\inputencoding{koi8-r}ðÒÅÄÐÏÌÏÖÅÎÉÅ}}
  \addto\captionsrussian{\renewcommand{\definitionname}{\inputencoding{koi8-r}ïÐÒÅÄÅÌÅÎÉÅ}}
  \addto\captionsrussian{\renewcommand{\lemmaname}{\inputencoding{koi8-r}ìÅÍÍÁ}}
  \addto\captionsrussian{\renewcommand{\remarkname}{\inputencoding{koi8-r}úÁÍÅÞÁÎÉÅ}}
  \providecommand{\conjecturename}{Conjecture}
  \providecommand{\definitionname}{Definition}
  \providecommand{\lemmaname}{Lemma}
  \providecommand{\remarkname}{Remark}
\addto\captionsenglish{\renewcommand{\theoremname}{Theorem}}
\addto\captionsngerman{\renewcommand{\theoremname}{Theorem}}
\addto\captionsrussian{\renewcommand{\theoremname}{\inputencoding{koi8-r}ôÅÏÒÅÍÁ}}
\providecommand{\theoremname}{Theorem}

\begin{document}

\title{Perturbing PLA}

\author{Gady Kozma and Alexander Olevskii}

\thanks{Both authors partially supported by their respective Israel Science
Foundation grants.}
\begin{abstract}
We proved earlier that every measurable function on the circle, after
a uniformly small perturbation, can be written as a power series (i.e.\ a
series of exponentials with positive frequencies), which converges
almost everywhere. Here we show that this result is basically sharp:
the perturbation cannot be made smooth or even H\"older. We discuss
also a similar problem for perturbations with lacunary spectrum.

\end{abstract}
\maketitle

\section{Introduction}

\subsection{Functions representable by analytic sums}

\def\plahat#1{\stackrel{\,\mbox{\fontsize{4}{0}\selectfont \raisebox{-0.4ex}{PLA}}}{\widehat{#1}}\hspace{-0.7ex}}
Let a power series converge almost everywhere on the circle $\mathbb{T}$
to a function $g$:
\begin{equation}
g(t)=\sum_{n\ge0}c(n)e^{int}\label{eq:e1}
\end{equation}
It follows from the Privalov uniqueness theorem, that any $g$ may
have at most one such decomposition. An analogy with the classical
Riemannian theory suggests that $c(n)$ are the Fourier coefficients,
whenever $g$ is integrable.

Quite surprisingly, this is not the case: a few years ago we constructed
an $L^{2}$-function $g$ on $\mathbb{T}$ which admits the representation
(\ref{eq:e1}) but 
\[
\sum|c(n)|^{2}=\infty.
\]
Later we proved that such a function even can be smooth.

The space of functions $g$ which admit an ``analytic'' representation
(\ref{eq:e1}) we named PLA. The classic PLA-part of $L^{2}(\mathbb{T})$
is the set of functions whose Fourier ser\-ies contains exponentials
with non-negative frequencies only, namely the Hardy space $H^{2}$.
This set is ``small'', in paticular it is nowhere dense. In contrast
the ``non-classic'' part is dense. Moreover, the following equality
is true:
\begin{equation}
L^{0}=\PLA+C(\mathbb{T})\label{eq:e2}
\end{equation}
which means that every measurable finite function $f$ can be decomposed
as a sum 
\begin{equation}
f=g+h\label{eq:e3}
\end{equation}
where $g\in\PLA$ and $h$ is continuous. Further, one can replace
$C(\mathbb{T})$ with $U(\mathbb{T})$, the space of uniformly convergent
Fourier series, and one can require from $h$ to have arbitrarily
small norm (the norm in $U(\mathbb{T})$ being the supremum of the
modulus of the partial sums of the Fourier expansion). The described
results are proved in \cite{KO.06,KO.07}.

Our first result is that the equality (\ref{eq:e2}) is close to best
possible: one can not replace the second summand by a space of functions
which possess any smoothness, like H\"older or Sobolev one. We state
the result in the following form. Given a sequence $\omega=\{\omega(n)\}$,
$0<\omega(n)\nearrow\infty$, denote:
\begin{equation}
\mathcal{H}_{\omega}=\left\{ h:\sum|\widehat{h}(n)|^{2}\omega^{2}(n)<\infty\right\} .\label{eq:e4}
\end{equation}

\begin{thm}
\label{thm:PLA+Hw}For any $\omega$ the sum $\PLA+\mathcal{H}_{\omega}$
does not cover neither $L^{0}(\mathbb{T})$, nor even the Wiener algebra
$A(\mathbb{T})=\widehat{l_{1}(\mathbb{Z})}$.
\end{thm}
This theorem will be proved in \S\ref{sec:ProofPLA+Hw}.

\subsection{Menshov spectra revisited.~~~}

The classic Menshov representation theorem (1940), see \cite[\S XV.2]{B64}
says that every function $f\in L^{0}(\mathbb{T})$ admits representation
by a trigonometric series which converges a.e.:
\begin{equation}
f=\sum_{n=-\infty}^{\infty}c(n)e^{int}.\label{eq:2_1}
\end{equation}
This representation is non-unique, as follows from another remarkable
result of Menshov's proved much earlier (1916): there is a non-trivial
trigonometric series which converges to zero almost everywhere. Menshov's
construction for the representation reveals the non-uniqueness phenomenon
in a stronger form: one can avoid in (\ref{eq:2_1}) using any finite
and even some infinite sets of harmonics. This leads to the following
definition, see \cite{KO.01}:
\begin{defn*}
A sequence $\Lambda\subset\mathbb{Z}$ is called a Menshov spectrum
if every function $f\in L^{0}(\mathbb{T})$ can be decomposed to a
series (\ref{eq:2_1}) in which only frequencies from $\Lambda$ may
appear with non-zero amplitudes.
\end{defn*}
In this terminology, Menshov's theorem states that $\mathbb{Z}$ is
a Menshov spectrum. There are many results that show that Menshov
spectra could be quite sparse. For example Arutyunyan \cite{A85}
showed that any symmetric set which contains arbitrarily long intervals
is a Menshov spectrum. In other words, the set
\[
\bigcup_{n=1}^{\infty}[a_{n},a_{n}+n]\cup[-a_{n}-n,-a_{n}]
\]
is a Menshov spectrum, no matter how fast do the $a_{n}$ grow. Of
course, such sets can be extremely sparse. Here we wish to compare
to the following sparseness result, taken from \cite{KO.01}:\medskip{}

\noindent \emph{Given a sequence 
\begin{equation}
\omega(k)=o(1)\label{eq:2_2}
\end{equation}
one can construct a sequence $\lambda(k)\in\mathbb{Z}^{+}$ with $\lambda(k+1)/\lambda(k)>1+\omega(k)$,
such that $\Lambda=\{\pm l(k)\}$ is a Menshov spectrum.\medskip{}
}

The condition (\ref{eq:2_2}) is sharp: a Menshov spectrum cannot
be lacunary in Hada\-mard sense. Further, the symmetry condition
is also essential. Indeed, Privalov's uniqueness theorem implies that
the set $\mathbb{Z}^{+}$ is not a Menshov spectrum. See \cite{KO.01}
for details on all these claims.

One may now ask: how many negative frequencies one should add to $\mathbb{Z}^{+}$
in order to get a Menshov spectrum? According to the theorem above
an extra set with gaps of any sub-exponential growth could be sufficient.
Our second result is that this result is close to the best possible
one: super-exponential growth is not sufficient.
\begin{thm}
\label{thm:notMS}Let $Q:=\{q(k)\}\subset\mathbb{Z}^{+}$ satisfy
the condition 
\begin{equation}
\frac{q(k+1)}{q(k)}\to\infty.\label{eq:2_3}
\end{equation}
Then the set $\Lambda=\mathbb{Z}^{+}\cup\{-Q\}$ is not a Menshov
spectrum.
\end{thm}
Theorem \ref{thm:notMS} can be reformulated in the language of theorem
\ref{thm:PLA+Hw}. Let 
\[
\mathcal{L}_{Q}=\{f\in L^{1}:\widehat{f}(n)=0\;\forall n\not\in Q\}.
\]
Then 

\begin{thm3p}With the same $Q$ as in theorem \ref{thm:notMS}, $\PLA+\mathcal{L}_{-Q}\ne L^{0}$.\end{thm3p}

The equivalence of theorems \ref{thm:notMS} and \ref{thm:notMS}'
follows by taking the Menshov representation of $f$ and making the
positive part into a PLA function and the negative part into an $\mathcal{L}_{-Q}$
function. This requires Plessner's theorem and some standard facts
on lacunary trigonometric series --- we fill these details in \S\ref{sec:Proof-Lacunary}.

The formulation of theorem \ref{thm:notMS}' leads to a natural generalisation.
Can one find a function $f\not\in\PLA+\mathcal{L}_{-Q}$ for all superexponential
$Q$ simultaneously? We present a weakned version of this
\begin{thm}
\label{thm:PLA+UL}For a function $\ell(n)\to\infty$ there is a
function $f$ such that $f\not\in\PLA+\mathcal{L}_{-Q}$ for any $Q$
satisfying
\[
\frac{q(k+1)}{q(k)}>\ell(q(k)).
\]

\end{thm}
We remark that, as in theorem \ref{thm:PLA+Hw}, the $f$ of theorem
\ref{thm:PLA+UL} may be taken to be in the Weiner algebra $A(\mathbb{T})$.

Theorem \ref{thm:PLA+UL} is clearly stronger than theorem \ref{thm:notMS}',
and hence also from theorem \ref{thm:notMS}. On the other hand, the
proof is also more technical. Hence we first prove theorem \ref{thm:notMS}
in \S\ref{sec:Proof-Lacunary}, and only afterwards give the proof
of theorem \ref{thm:PLA+UL} in \S\ref{sec:Proof-lac-simul}.

\section{Lemmas}

In this section we introduce some notation and lemmas which will be
used for the proof of all three theorems. For a PLA-function $g$
the coefficients $c_{n}$ in the expansion (\ref{eq:e1}) are unique,
so we will denote them by $\plahat{g}(n)$. Below we denote by $g$
any PLA-function with $\plahat{g}(0)=0$. We will use the following
notations\smallskip{}

\begin{minipage}[c][1\totalheight][t]{0.2\columnwidth}%
\input{priv3.pstex_t}\end{minipage}\hfill{}%
\begin{minipage}[c][1\totalheight][t]{0.75\columnwidth}%
\begin{align}
g^{*}(t) & :=\sup_{N}\Big|\sum_{n<N}\plahat{g}(n)e^{int}\Big|\nonumber \\
G(z) & :=\sum\plahat{g}(n)z^{n},|z|<1\label{eq:defG}\\
Q_{t} & :=\conv(\{e^{it}\}\cup\{z:|z|<\tfrac{1}{2}\}).\nonumber 
\end{align}
\end{minipage}\medskip{}

\noindent This $Q_{t}$ is often called the Privalov ice-cream cone
at $e^{it}$. We always denote by $E$ a measurable subset of $\mathbb{T}$;
by $|E|$ its Lebesgue measure. By $\left\Vert \cdot\right\Vert _{2}$
we denote the norm in $L^{2}(\mathbb{T})$.
\begin{lem}
\label{lem:bounded cone}If $g^{*}(t)=A$ then $|G(z)|\le3A$ for
all $z\in Q_{t}$.\end{lem}
\begin{proof}
Without loss of generality one may assume $t=0$. We now apply Abel's
summation formula to get, for any $|z|<1$,
\[
G(z)=\sum_{n=0}^{\infty}\plahat{g}(n)z^{n}=\sum_{n=0}^{\infty}(z^{n}-z^{n+1})\sum_{k=0}^{n}\plahat{g}(k)
\]
so
\[
|G(z)|\le A\sum_{n=0}^{\infty}|z^{n}-z^{n+1}|=\frac{A|1-z|}{1-|z|}
\]
but in $Q_{0}$ one has $|1-z|/(1-|z|)\le3$, with the maximum achieved
at $z=-\frac{1}{2}$ (the exact value of the constant 3 will play
no role in what follows).\end{proof}
\begin{lem}
\label{lem:main}There is some universal constant $c_{1}>0$ such
that for every $K>0$ there is a number $\eps=\eps(K)$ such that
if 
\[
\left\Vert 1+g\right\Vert _{2}<\eps
\]
then
\[
|\{t:g^{*}(t)>K\}|>c_{1}.
\]

\end{lem}
We remark that in fact $c_{1}$ may be taken to be $\nicefrac{1}{2}$
or any number smaller than $1$, but this requires an extra argument
that we prefer to skip. The dependency between $\varepsilon$ and
$K$ will turn out to be $\varepsilon\approx\nicefrac{1}{K^{2}}$.
The power can be reduced arbitrarily close to zero (e.g.\ $\varepsilon\approx K^{-0.0001}$)
at the price of making $c_{1}$ smaller (we will not need all these
in this paper).
\begin{proof}
The proof is a simple variation on the proof of Privalov's uniqueness
theorem \cite[\S D.III]{K80}. Let $A>1$ be some sufficiently large
parameter to be fixed later, and let $E\subset[0,2\pi]$ be the set
of $t$ satisfying the following two requirements
\begin{gather}
\begin{aligned}|1+g(t)| & <A\eps &  & \forall t\in E\\
|g^{*}(t)| & \le K &  & \forall t\in E.
\end{aligned}
\label{eq:e5}
\end{gather}
Assume by contradiction that $|\{t:g^{*}(t)>K\}|\le c_{1}$. Then
we may assume that 
\[
|E|>1-\frac{1}{A^{2}}-c_{1}
\]
since Markov's inequality gives us 
\[
|\{t:|1+g(t)|\ge A\eps\}|<\frac{1}{A^{2}}.
\]
(the fact that the power is $2$ will play no role in the argument).

Next, recall that $G$ is the ``extension'' of $g$ into the disk
$\{|z|\le1\}$ defined by (\ref{eq:defG}) whenever the sum converges,
which is on all of $\{|z|<1\}$ and almost everywhere on $\{|z|=1\}$,
since $g$ is in PLA. By Abel's theorem, if $\sum\plahat{g}(n)e^{int}$
converges then $G(z)\to G(e^{it})$ when $z$ converges to $e^{it}$
non-tangentially \cite[\S 3.14]{Z68}. Assume therefore, without loss
of generality, that the convergence $G(z)\to G(e^{it})$ is uniform
on $E$ and that $E$ is closed (if it is not, use Egoroff's theorem
to find an $E'\subset E$ satisfying the requirement and having large
measure, $|E'|>1-A^{-2}-c_{1}$). 

Examine the Privalov domain over $E$, namely
\[
P=\bigcup_{t\in E}Q_{t}.
\]
See figure \ref{fig:Priv} which demonstrates a Privalov domain for
a Cantor set. By the above, $G$ is continuous on $P$. Next examine
the function $\ell:=\log|1+G|$. It is subharmonic on $P$, and continuous
on $\overline{P}$ (in the sense that allows the value $-\infty$).
Therefore
\begin{equation}
\ell(0)\le\int_{\partial P}\ell(z)d\Omega(z)\label{eq:subharmonic}
\end{equation}
where $\Omega$ is the harmonic measure of $P$ from $0$ (which is
clearly a point of $P$). For background on the harmonic measure (and
especially its construction using Brownian motion) see the book \cite{B95}.

To use (\ref{eq:subharmonic}) we first note that $G(0)=0$ so $\ell(0)=0$.
Now examine the boundary of $P$. We write $\partial P=E\cup I$.
On $E$ we have $\ell<\log A\eps$. On $I$ we apply lemma \ref{lem:bounded cone}
to see that $\ell\le\log3K$. Finally we need to estimate the harmonic
measure $\Omega$ of $I$ in the domain $P$. Every interval $J$
in the complement of $E$ corresponds to a piece $J'$ of $I$ ---
usually to just two straight lines from the edges of $e^{iJ}$, but
sometimes also to a piece of $\{|z|=\frac{1}{2}\}$. One $J$ and
$J'$ of the second kind (i.e.\ with a piece of $\{|z|=\frac{1}{2}\}$)
are noted in figure \ref{fig:Priv}. 
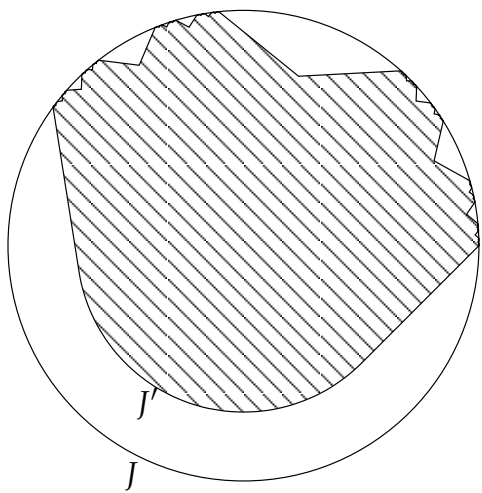
\begin{figure}
\input{priv.pstex_t}

\caption{A Privalov domain for a Cantor set.}
\label{fig:Priv}
\end{figure}
 Either way, a straightforward calculation shows that the probability
that Brownian motion starting from $0$ hits $J'$ before leaving
the disk is $\le C|J|$ and hence 
\[
\Omega(I)=\sum\Omega(J')\le C\sum|J|=C|E^{c}|<C(A^{-2}+c_{1}).
\]
Define therefore
\[
A=2C^{-1/2}\qquad c_{1}=\frac{1}{4C}
\]
and get that $\Omega(I)<\frac{1}{2}$ and hence that $\Omega(E)>\frac{1}{2}$.
With the estimates above for $\ell$ and (\ref{eq:subharmonic}) we
get 
\[
0=\ell(0)\le\Omega(E)\log A\eps+\Omega(I)\log3K<\frac{1}{2}\log A\eps+\log3K
\]
which leads to a contradiction if only $\epsilon$ is sufficiently
small. \end{proof}
\begin{rem*}
It might be worthwhile to compare this lemma to lemma 2.6 in \cite{N93},
which is also proved by Privalov's approach.
\end{rem*}

\section{\label{sec:ProofPLA+Hw}Proof of theorem \ref{thm:PLA+Hw}.}

Let $\{b(k)\}$ ($k=1,2,\dotsc$) be a fast decreasing sequence of
positive numbers, such that 
\begin{equation}
\sum_{j>k}b(j)=o\left(b(k)\eps\left(\frac{k}{b(k)}\right)\right)\qquad\forall k\label{eq:e6}
\end{equation}
where $\eps(K)$ are from lemma \ref{lem:main}.

Recall the sequence $\omega$ going to infinity from the statement
of the theorem. Given $\omega$, choose a fast increasing sequence
$\{n(k)\}$ of integers so that 
\begin{equation}
b(k)\eps\left(\frac{k}{b(k)}\right)\omega(n(k))\to\infty\label{eq:e7}
\end{equation}
Set
\[
f(t):=\sum_{k>0}b(k)e^{-in(k)t}.
\]
We claim that the function $f$ does not belong to $\PLA+\mathcal{H}_{\omega}$.
Take therefore any $h\in\mathcal{H}_{\omega}$ and let $q:=f-h$.
We need to show that $q\not\in\PLA$. Fix a (large) number $N$. Denote
\begin{align*}
f'(N;t) & =\sum_{k<N}b(k)e^{-in(k)t}\\
f''(N;t) & =\sum_{k>N}b(k)e^{-in(k)t}
\end{align*}
so 
\[
f=f'+b(N)e^{-in(N)t}+f".
\]
Similarly, let 
\[
h=h'+h"
\]
where 
\[
h':=\sum_{|n|<n(N)}\widehat{h}(n)e^{int}.
\]
Clearly 
\[
|f"(t)|\le\sum_{j>k}b(j)
\]
and
\begin{equation}
\left\Vert h"\right\Vert _{2}<\left\Vert h\right\Vert _{\mathcal{H}_{\omega}}/\omega(n(N)).\label{eq:hpp<hw}
\end{equation}

Examine now the function $g:=[1/b(N)](f'-h'-q)e^{in(N)t}$. It is
in PLA since $q\in\PLA$ (we argue by contradiction here) and $e^{in(N)t}(f'-h')$
is an analytic polynomial. In other words
\[
\plahat{g}(k)=\frac{1}{b(N)}\left(\widehat{f'}(k-n(N))-\widehat{h'}(k-n(N))-\plahat{q}(k-n(N))\right)
\]
and $\plahat{g}(0)=0$. Hence we may apply lemma \ref{lem:main}.
For the $L^{2}$ norm we can write
\begin{align*}
\left\Vert 1+g\right\Vert _{2} & =\left\Vert \frac{e^{in(N)t}}{b(N)}\left(b(N)e^{-in(N)t}+f'-h'-q\right)\right\Vert _{2}=\\
 & =\frac{1}{b(N)}\left\Vert -f''+h''\right\Vert _{2}\le\\
\mbox{By (\ref{eq:e6}) and (\ref{eq:hpp<hw})}\qquad & \le\frac{1}{b(N)}\left(o\left(b(N)\eps\left(\frac{N}{b(N)}\right)\right)+\frac{||h||_{\mathcal{H}_{\omega}}}{\omega(n(N))}\right)\\
\mbox{By (\ref{eq:e7})}\qquad & =o\left(\eps\left(\frac{N}{b(N)}\right)\right)
\end{align*}
(where the $o$ is allowed to depend on $||h||_{\mathcal{H}_{\omega}}$).
So for $N$ sufficiently large the $o$ is smaller than $1$, and
the lemma gives that 
\begin{equation}
\left|\left\{ t:g^{*}(t)>\frac{N}{b(N)}\right\} \right|>c_{1}.\label{eq:g*large}
\end{equation}

At this point we only need to go back from $g$ to $q$, so we need
to estimate the contributions of $f'$ and $h'$. $f'$ is straightforward
as
\begin{equation}
\sup_{k}\left|\sum_{j<k}\widehat{f'e^{in(N)t}}(j)e^{ijt}\right|\le\sum_{k}|\widehat{f}(k)|<C.\label{eq:fprime}
\end{equation}
For $h'$ we use Carleson's theorem \cite{C66,L04} for both $h^{+}$
and $h^{-}$ defined by 
\[
h^{+}=\sum_{n\ge0}\widehat{h}(n)e^{int}\qquad h^{-}=\sum_{n<0}\widehat{h}(n)e^{int}
\]
and get that both expansions converge almost everywhere. This gives
a set $E$ with $|E^{c}|\le\frac{1}{2}c_{1}$ such that
\begin{equation}
\left|\sum_{n=0}^{k}\widehat{h}(n)e^{int}\right|\le C\qquad\left|\sum_{n=-k}^{-1}\widehat{h}(n)e^{int}\right|\le C\qquad\forall t\in E,\forall k.\label{eq:Carleson}
\end{equation}
For $h'e^{in(N)t}$, the analogous sum is bounded by either a sum
of two terms from (\ref{eq:Carleson}), or by a difference of two,
and in both cases we get 
\begin{equation}
\left|\sum_{j<k}\widehat{h'e^{in(N)t}}(j)e^{ijt}\right|\le2C.\label{eq:hprime}
\end{equation}
This proves the theorem: since $g^{*}$ is large (\ref{eq:g*large})
and $f'$ and $h'$ are bounded (\ref{eq:fprime}), (\ref{eq:hprime}),
we get
\[
q^{*}\ge N-C
\]
on a set of measure $>\frac{1}{2}c_{1}$. Since $N$ was arbitrary
and $C$ depends only on $h$, this proves that $q\not\in\PLA$.\qed

 \begin{qst*}Does $\PLA+A(\mathbb{T})$ cover $C(\mathbb{T})$?
\end{qst*}

\section{\label{sec:Proof-Lacunary}Proof of theorem \ref{thm:notMS}}

Let us start by showing that theorem \ref{thm:notMS} is equivalent
to theorem \ref{thm:notMS}' (this will also aid in its proof). For
this we need two classical results
\begin{enumerate}
\item \noindent \emph{\label{enu:lacunary}A lacunary trigonometric sum
converges almost everywhere if and only if it is in $L^{2}$. A function
in $L^{1}$ with a lacunary Fourier expansion is in $L^{2}$. }See
e.g.\ \cite{Z68} \S 5.6. Here lacunary means in Hadamard sense,
i.e.\ $q(k+1)/q(k)>1+c$.
\item \emph{\label{enu:Plessner}If a trigonometric series converges pointwise
on a set $E$, then both its positive and negative parts converge
almost everywhere on $E$.} This result is due to Plessner \cite{P25}.
A careful treatment can be found in \cite{B64}, \S VIII.23, volume
2, page 151.
\end{enumerate}
To see that theorem \ref{thm:notMS} implies theorem \ref{thm:notMS}'
note that a function $f$ which proves that $\Lambda$ is not a Menshov
spectrum also cannot be in $\PLA+\mathcal{L}_{-Q}$ as a decomposition
$f=g+h$, $g\in\PLA$, $h\in\mathcal{L}_{-Q}$ carries over to a representation
\[
f(t)=\sum_{n=0}^{\infty}\plahat{g}(n)e^{int}+\sum_{n=-\infty}^{0}\widehat{h}(n)e^{int}
\]
which converges almost everywhere since both its parts converge almost
everywhere: the $g$ part by definition of PLA and the $h$ part because
of (\ref{enu:lacunary}).

Vice versa, assume by contradiction that $\Lambda$ is a Menshov
spectrum. Than every $f$ has a representation as a sum $\sum_{n\in\Lambda}c(n)e^{int}$
converging almost everywhere. But by Plessner's theorem the positive
part converges a.e., so its limit, call it $g$, is a PLA function.
Also the negative part converges a.e.\ so call its limit $h$. By
the other direction of (\ref{enu:lacunary}), $h\in L^{2}$ and hence
in $\mathcal{L}_{-Q}$. We get $f=g+h$ with $g\in\PLA$ and $h\in\mathcal{L}_{-Q}$
so, since $f$ was arbitrary, $\PLA+\mathcal{L}_{-Q}=L^{0}$. This
shows that theorem \ref{thm:notMS}' implies theorem \ref{thm:notMS},
so they are equivalent. \qed

Satisfied that theorem \ref{thm:notMS} and \ref{thm:notMS}' are
equivalent, we start their proof. The first step is the following
simple lemma.
\begin{lem}
\label{lem:alpha}Let Q satisfy (\ref{eq:2_3}). Then there is a number
$\alpha$, $\frac{1}{3}<\alpha<\frac{2}{3}$ such that
\begin{equation}
\{\alpha q(k)\}=o(1)\label{eq:2_4}
\end{equation}
where $\{x\}$ denotes the fractional part of $x$.
\end{lem}
We remark that in \S \ref{sec:Proof-lac-simul} we will need that
the estimate of $\{\alpha q\}$ can be done uniformly in the superexponential
growth of the $Q$, namely,
\begin{equation}
\{\alpha q(k)\}\le C\max_{l\ge k}\left\{ \frac{q(l)}{q(l+1)}\right\} .\label{eq:nomoremu}
\end{equation}
Also the restriction $\alpha\in(\frac{1}{3},\frac{2}{3})$ is only
used in \S \ref{sec:Proof-lac-simul}, here $\alpha$ can be taken
anywhere in $(0,1)$.
\begin{proof}
We may assume without loss of generality that $q(k+1)/q(k)>2$ for
all $k$ (for $k=1$ we assume $q(1)>2$). Set
\begin{equation}
\alpha=\frac{1}{3}+\sum_{k=1}^{\infty}\frac{\gamma(k)}{q(k)}\label{eq:2_5}
\end{equation}
where the numbers $0<\gamma(k)\leq1$ are to be defined. Assuming
they are already defined for $k\le n$, we denote by $a(n)$ the $n^{\textrm{th}}$
partial sums of the series (\ref{eq:2_5}) and set $\gamma(n+1):=1-\{a(n)q(n+1)\}$
which implies that $q(n+1)a(n+1)$ is integer. Continuing this process
we get $\alpha$.

Now, for every $n>1$:
\[
\alpha q(n)=a(n)q(n)+q(n)\sum_{k>n}\frac{\gamma(k)}{q(k)}.
\]
As already explained, $a(n)q(n)$ is an integer. The second term is
$\le q(n)\sum_{k>n}\linebreak[4]1/q(k)$, which is $o(1)$ due to
(\ref{eq:2_3}). This gives (\ref{eq:2_4}) and also the  uniform
estimate (\ref{eq:nomoremu}) remarked upon after the lemma.
\end{proof}
\noindent \textbf{Step 1}. With the lemma proved we can start the
proof of theorem \ref{thm:notMS}. Fix numbers $d(n)>0$ decreasing
so fast that
\begin{equation}
\sum_{n>N}d^{2}(n)<\frac{d(N)^{2}}{N^{2}}\varepsilon^{2}\left(\frac{N}{d(N)}\right)\;,\label{eq:2_6-1}
\end{equation}
where the function $\varepsilon(n)$ was defined in lemma \ref{lem:main}.
Since this expression will repeat a lot, we will denote it for short
by $\varepsilon_{N}$, 
\begin{equation}
\varepsilon_{N}:=\frac{d(N)}{N}\varepsilon\left(\frac{N}{d(N)}\right)\;,\label{eq:defepseps-1}
\end{equation}
so $\sum_{n>N}d^{2}(n)<\varepsilon_{N}^{2}$. \medskip{}

\noindent \textbf{Step 2}. Next use lemma \ref{lem:alpha} to find
a number $\beta\in(0,2\pi)$ such that
\begin{equation}
e^{i\beta q(k)}\to1.\label{eq:orig-7}
\end{equation}

\noindent \textbf{Step 3}. With $\beta$ defined, one can find $\nu(N)$
such that the following two properties hold,
\begin{align}
|1-e^{i\beta q(k)}| & <\varepsilon_{N}\qquad\forall k\mbox{ such that }q(k)>\nu(N)\label{eq:orig-8}\\
|1-e^{i\beta\nu(N)}| & >1.\label{eq:orig-9}
\end{align}
These properties can be satisfied simultaneously because (\ref{eq:orig-8})
is satisfied whenever $\nu(N)$ is sufficiently large, while (\ref{eq:orig-9})
is satisfied on a sequence converging to $\infty$. We now define
\begin{equation}
f(t)=\sum_{N=1}^{\infty}d(N)e^{-i\nu(N)t}.\label{eq:orig-10}
\end{equation}
This is the required function. As an aside we remark that it is in
the Wiener algebra, but it might be highly non-smooth as we have no
control over the relation between $d(N)$ and $\nu(N)$.\medskip{}

\noindent \textbf{Step 4}. Recall now the discussion in the beginning
of this section. We claim that $f$ is a function demonstrating that
$\Lambda=\{-Q\}\cup\mathbb{Z}^{+}$ is not a Menshov spectrum, i.e.\ that
$f$ has no expansion
\begin{equation}
f(t)=\sum_{n}c(n)e^{int}\qquad n\not\in\Lambda\implies c(n)=0\label{eq:thm2pp}
\end{equation}
which converges almost everywhere. Assume therefore by contradiction
that an expansion (\ref{eq:thm2pp}) exists. Due to Plessner's theorem
we know that $\sum_{n<0}c(n)e^{int}$ converges (to some value), and
since the negative part is lacunary we must have $\sum_{n<0}|c(n)|^{2}<\infty$.

Somewhat similarly to the proof of theorem \ref{thm:PLA+Hw}, we will
now subtract the Fourier expansion of $f$ and the non-standard one
(\ref{eq:thm2pp}) and get a null series i.e.\ a trigonometric series
converging to zero almost everywhere. Namely, define
\[
\gamma(n)=c(n)-\widehat{f}(n)=c(n)-\begin{cases}
d(N) & n=-\nu(N)\\
0 & \mbox{otherwise}
\end{cases}
\]
and get that 
\[
\sum_{n=-\infty}^{\infty}\gamma(n)e^{int}=0\quad\mbox{for almost every }t.
\]
The crucial step is to examine $f(t+\beta)-f(t)$ and the corresponding
null series
\begin{equation}
\sum_{n=-\infty}^{\infty}\gamma(n)(e^{in\beta}-1)e^{int}=0\quad\mbox{for almost every }t.\label{eq:nullbeta-1}
\end{equation}
As in the remark after (\ref{eq:thm2pp}), the positive and negative
parts of (\ref{eq:nullbeta-1}) converge almost everywhere (not necessarily
to zero).

\medskip{}

\noindent \textbf{Step 5}. We will need some estimates for the $L^{2}$
norm of the ``tails'' of (\ref{eq:nullbeta-1}) so let us state
them now: for every $N$, 
\begin{equation}
\sum_{n<-\nu(N)}|\gamma(n)(e^{i\beta n}-1)|^{2}\le C\varepsilon_{N}^{2}.\label{eq:gamman}
\end{equation}
Here and below $C$ may depend on $\sum_{n<0}|c(n)|^{2}$ (but not
on $N$).
\begin{proof}
[Proof of (\ref{eq:gamman})]$\gamma(n)$ can be non-zero only if
$n=-q(k)$ or if $n=-\nu(k)$. In the first case we have
\[
|e^{i\beta n}-1|=|e^{-i\beta q(k)}-1|=|e^{i\beta q(k)}-1|\stackrel{(\textrm{\ref{eq:orig-8}})}{<}\varepsilon_{N}
\]
(recall that we are looking at $n<-\nu(N)$ so (\ref{eq:orig-8})
applies). All in all this gives
\[
\sum_{k:q(k)>\nu(N)}|c(-q(k))(e^{i\beta q(k)}-1)|^{2}<\varepsilon_{N}^{2}\sum_{n<0}|c(n)|^{2}
\]
which we agreed to denote by $C\varepsilon_{N}^{2}$. The second kind
of non-zero $n$ is $-\nu(k)$ and for this we simply use the definition
of the $d(k)$, (\ref{eq:2_6-1}) and of $f$, (\ref{eq:orig-10}),
and get
\[
\sum_{k>N}d(k)^{2}|e^{i\beta\nu(k)}-1|^{2}\le4\sum_{k>N}d(k)^{2}\stackrel{(\textrm{\ref{eq:2_6-1}})}{\le}4\varepsilon_{N}^{2}.
\]
Taking these two estimates together gives
\begin{align*}
\sum_{n<-\nu(N)}|\gamma(n)(e^{i\beta n}-1)|^{2} & =\sum_{n<-\nu(N)}|(c(n)+\widehat{f}(n))(e^{i\beta n}-1)|^{2}\le\\
 & \le\sum_{n<-\nu(N)}(2|c(n)|^{2}+2|\widehat{f}(n)|^{2})\cdot|e^{i\beta n}-1|^{2}\\
\mbox{By the above}\qquad & \le\varepsilon_{N}^{2}\cdot\left(2C+8\right)
\end{align*}
as needed.
\end{proof}
\noindent \textbf{Step 6}. We now proceed as in the proof of theorem
\ref{thm:PLA+Hw} i.e.\ we wish to apply lemma \ref{lem:main} for
some PLA function related to the null-series (\ref{eq:nullbeta-1}).
We shift the null-series (\ref{eq:nullbeta-1}) by $\nu(N)$ and divide
it by $d(N)(e^{-i\nu(N)\beta}-1)$. We get
\begin{align}
q(t): & =\frac{1}{d(N)(e^{-i\nu(N)\beta}-1)}\sum_{n>-\nu(N)}\gamma(n)(e^{i\beta n}-1)e^{i(n+\nu(N))t}\label{eq:simple-qhat}\\
 & =1+\frac{1}{d(N)(e^{-i\nu(N)\beta}-1)}\sum_{n<-\nu(N)}\gamma(n)(e^{i\beta n}-1)e^{i(n+\nu(N))t}.\nonumber 
\end{align}
In other words, the first line is the ``PLA expansion'' of $q$
and the second is the Fourier expansion. In particular $q$ is a PLA
function with $\plahat{q}(0)=0$. By (\ref{eq:gamman}), 
\[
||1-q||_{2}\le\frac{C\varepsilon_{N}}{d(N)|e^{-i\nu(N)\beta}-1|}.
\]
By requirement (\ref{eq:orig-9}), $|e^{-i\nu(N)\beta}-1|>1$, and
with the definition of $\varepsilon_{N}$ we get
\[
||1-q||_{2}\le\frac{C}{N}\varepsilon\left(\frac{N}{d(N)}\right).
\]
Hence for $N>C$ we may apply lemma \ref{lem:main} (to $-q$, but
$q^{*}=(-q)^{*}$) and get
\[
\left|\left\{ t:q^{*}(t)>\frac{N}{d(N)}\right\} \right|>c_{1}.
\]
Recalling that the PLA expansion of $q$ is (\ref{eq:simple-qhat})
we get a set of measure $>c_{1}$ where 
\begin{equation}
\sup_{k}\left|\sum_{n=-\nu(N)}^{k}\gamma(n)(e^{i\beta n}-1)e^{i(n+\nu(N))t}\right|>d(N)|e^{-i\nu(N)\beta}-1|\frac{N}{d(N)}>N.\label{eq:sumnuNk}
\end{equation}

\noindent \textbf{\medskip{}
Step 7}. We only need to change the lower bound in the sum. But clearly
\[
\sum_{n=-\nu(N)}^{0}|\gamma(n)(e^{i\beta n}-1)|^{2}\le C
\]
(if you want you can deduce this from (\ref{eq:gamman}) with the
$N$ there being 0). Using Markov's inequality gives that the corresponding
function cannot be large on a set of large measure: 
\[
\left|\left\{ t:\Bigg|\smash{\sum_{n=-\nu(N)}^{0}}\gamma(n)(e^{i\beta n}-1)e^{int}>C\Bigg|\right\} \right|<\frac{1}{2}c_{1}.
\]
We subtract this from (\ref{eq:sumnuNk}) and get a set of measure
$>\frac{1}{2}c_{1}$ where
\[
\sup_{k}\left|\sum_{n=0}^{k}\gamma(n)(e^{i\beta n}-1)e^{int}\right|>N-C.
\]
Since $N$ was arbitrary, we get a set of measure $>\frac{1}{2}c_{1}$
where
\[
\sup_{k}\left|\sum_{n=0}^{k}\gamma(n)(e^{i\beta n}-1)e^{int}\right|=\infty.
\]
But this is exactly the positive part of (\ref{eq:nullbeta-1}). This
is a contradiction since it was supposed to converge almost everywhere.
This finishes the proof of theorem \ref{thm:notMS}.\qed
\begin{conjecture*}
Probably theorem \ref{thm:notMS}, and perhaps even theorem \ref{thm:PLA+UL},
hold for $Q$ lacunary in Hadamard sense.
\end{conjecture*}
There is another version of this problem. Let us introduce the concept
of ``Privalov spectrum''. We say that $\Lambda$ is a Privalov spectrum
if 
\[
\sum_{n\in\Lambda}c(n)e^{int}=0\qquad\forall t\in E,\quad|E|>0\implies c(n)\equiv0
\]
Clearly, a Menshov spectrum can never be a Privalov spectrum. The
trick of shifting by $\beta$ employed above is useful also for this
problem. For example, $\Lambda=\{-2^{n}\}_{n=1}^{\infty}\cup\mathbb{Z}^{+}$
is a Privalov spectrum. To see this, it is enough to shift by $\beta=2^{-k}$
with $k$ sufficiently large so as to satisfy $E\cap(E+\beta)\ne\emptyset$,
and this reduces the result to the original Privalov theorem. 

Thus a natural variation on the conjecture is: how sparse must $Q$
be in order to ensure that $-Q\cup\mathbb{Z}^{+}$ is a Privalov spectrum?
This problem was considered by F. Nazarov in the early 90s in an unpublished
work (private communication). The trick of shifting can be used to
show that if $Q$ is very fast increasing, then $-Q\cup\mathbb{Z}^{+}$
is Privalov, but it seems that not under the condition (\ref{eq:2_3})
of superexponential growth. Faster growth of $Q$ is necessary.

Another interesting generalization is to ask whether removing a superexponential
sequence from a Menshov spectrum leaves one with a Menshov spectrum.
Let us remark that a theorem of Talalyan \cite{T69} shows that removing
a single element from a Menshov spectrum will always result in a new
Menshov spectrum.

\section{\label{sec:Proof-lac-simul}Proof of theorem \ref{thm:PLA+UL}}

The $f$ demonstrating theorem \ref{thm:PLA+UL} cannot be exactly
as in the proof of theorem \ref{thm:notMS}, as that $f$ was lacunary
itself! Hence it is itself in some $\mathcal{L}_{-Q}$, without the
need to add any PLA function. It turns out that one can construct
an $f$ demonstrating theorem \ref{thm:PLA+UL} and very close to
lacunary. We will construct an $f\not\in\PLA+\mathcal{L}_{-Q}$ for
any $Q$ which is a sum of extremely lacunary \emph{couples} of consecutive
harmonics. The role of $\beta$ (the value you shift by in the proof)
in the theorem also changes --- it has to be chosen after $f$ is
already known, so $f$ cannot depend on it. \medskip{}

\noindent \textbf{Step 1}. To start the proof of theorem \ref{thm:PLA+UL},
we fix numbers $d(n)>0$ decreasing very fast. The precise condition
will not make much sense now, so please do not dwell on it: it will
become clearer in later stages of the proof. Precisely we define 
\begin{equation}
\varepsilon\varepsilon_{N}:=\frac{d(N)^{2}}{N^{2}}\varepsilon\left(\frac{N}{d(N)}\right)\varepsilon\left(\frac{N^{2}}{d(N)^{2}}\varepsilon^{-1}\left(\frac{N}{d(N)}\right)\right)\;,\label{eq:defepseps}
\end{equation}

\noindent and then require $d(n)$ to satisfy
\begin{equation}
\sum_{n>N}d^{2}(n)<\varepsilon\varepsilon_{N}^{2}.\label{eq:2_6}
\end{equation}
Comparing to (\ref{eq:2_6-1}) we see that instead of using the function
$\varepsilon(n)$ from lemma \ref{lem:main} once, as we did in (\ref{eq:2_6-1}),
here we need to iterate it. This is the reason for the notation $\varepsilon\varepsilon_{N}$.\medskip{}

\noindent \textbf{Step 2}. The choice of $\nu$ now cannot depend
on $\beta$ as it is not yet known --- it will instead depend on $\ell$,
the rate at which $q(k+1)/q(k)$ goes to infinity. Precisely, for
every $N$ find a $\nu(N)$ such that
\begin{equation}
\ell(\nu(N))>\frac{1}{\varepsilon\varepsilon_{N}}\label{eq:defnun}
\end{equation}
where  $\ell$ is from the statement of theorem \ref{thm:PLA+UL}.
We assume at this point that $\ell$ is increasing, which we may,
without loss of generality.\medskip{}

\noindent \textbf{Step 3}. With these we may define our function $f$,
\begin{equation}
f(t):=\sum_{n=1}^{\infty}d(n)\left[e^{-i(\nu(n)-1)t}+e^{-i\nu(n)t}\right].\label{eq:2_10}
\end{equation}

\noindent \textbf{Step 4}. We now need to show that $f\not\in\PLA+\mathcal{L}_{-Q}$,
for any $Q$. Assume to the contrary that $f=g+h$ with $g\in\PLA$
and $h\in\mathcal{L}_{-Q}$ for some $Q$ with $q(k+1)/q(k)>\ell(q(k))$.
As in the proof of theorem \ref{thm:notMS} we denote by $c(n)$ the
coefficients of this ``non-standard expansion'' of $f$, i.e.\ $c(n)=\plahat{g}(n)$
for $n\ge0$ and $c(n)=\widehat{h}(n)$ for $n<0$. Again we get a
null series by subtracting the Fourier expansion of $f$ and the non-standard
one. Namely, define
\[
\gamma(n)=c(n)-\begin{cases}
d(k) & n=-\nu(k)+1\mbox{ or }n=-\nu(k)\\
0 & \mbox{otherwise}
\end{cases}
\]
and get that 
\[
\sum_{n=-\infty}^{\infty}\gamma(n)e^{int}=0\quad\mbox{for almost every }t.
\]
Next apply lemma \ref{lem:alpha} (and the remark following it) to
find a number $\beta\in(\frac{2\pi}{3},\frac{4\pi}{3})$ such that
\begin{equation}
|e^{i\beta q(k)}-1|<\frac{C}{\ell(q(k))}\quad\forall k.\label{eq:2_7}
\end{equation}
(the $C$ has two sources: the first is (\ref{eq:nomoremu}) and the
second is the inequality $|e^{2\pi it}-1|\le C\{t\}$). As before,
the crucial step is to examine $f(t+\beta)-f(t)$ and the corresponding
null series
\begin{equation}
\sum_{n=-\infty}^{\infty}\gamma(n)(e^{in\beta}-1)e^{int}=0\quad\mbox{for almost every }t.\label{eq:nullbeta}
\end{equation}
Again this series not only converges to 0 symmetrically, also its
positive part $\sum_{n=0}^{\infty}$ and its negative part converge
individually, almost everywhere, for the same reasons as before. \medskip{}

\noindent \textbf{Step 5}. We will need estimates for the $L^{2}$
norm of the tails of (\ref{eq:nullbeta}), analogous to those of (\ref{eq:gamman}).
Precisely, 
\begin{equation}
\sum_{n<-\nu(N)}|\gamma(n)(e^{i\beta n}-1)|^{2}\le C\varepsilon\varepsilon_{N}^{2}.\label{eq:gammaneen}
\end{equation}
The proof is practically the same as that of (\ref{eq:gamman}), but
we include it for the convenience of the reader.
\begin{proof}
$c(n)$ can be non-zero only if $n=-q(k)$ or if $n=-\nu(k)+1$ or
$-\nu(k)$. In the first case we have
\[
|e^{i\beta n}-1|=|e^{-i\beta q(k)}-1|=|e^{i\beta q(k)}-1|\stackrel{(\textrm{\ref{eq:2_7}})}{<}\frac{C}{\ell(q(k))}.
\]
Now, we are looking at $n<-\nu(N)$ so by the definition of $\nu(N)$,
(\ref{eq:defnun}), 
\[
\ell(q(k))\ge\ell(\nu(N))\stackrel{\textrm{(\ref{eq:defnun})}}{<}\frac{1}{\varepsilon\varepsilon_{N}}.
\]
All in all this gives
\[
\sum_{l:q(l)>\nu(k)}|c(-q(l))(e^{i\beta q(l)}-1)|^{2}<C\varepsilon\varepsilon_{N}^{2}.
\]
The second kind of non-zero $n$ is $-\nu(k)+1$ and $-\nu(k)$ and
for this we simply use the definition of the $d(k)$, (\ref{eq:2_6})
and of $f$, (\ref{eq:2_10}), and get
\[
\sum_{k>N}d(k)^{2}\left(|e^{i\beta\nu(k)}-1|^{2}+|e^{i\beta(\nu(k)+1)}-1|^{2}\right)\le8\sum_{k>N}d(k)^{2}\stackrel{(\textrm{\ref{eq:2_6}})}{\le}8\varepsilon\varepsilon_{N}^{2}.
\]
Taking these two estimates together gives
\begin{align*}
\sum_{n<-\nu(N)}|\gamma(n)(e^{i\beta n}-1)|^{2} & =\sum_{n<-\nu(N)}|(c(n)+\widehat{f}(n))(e^{i\beta n}-1)|^{2}\le\\
 & \le\sum_{n<-\nu(N)}(2|c(n)|^{2}+2|\widehat{f}(n)|^{2})\cdot|e^{i\beta n}-1|^{2}\\
\mbox{By the above}\qquad & \le\varepsilon\varepsilon_{N}^{2}\cdot\left(2C+16\right).\qedhere
\end{align*}

\end{proof}
\noindent \textbf{Step 6}. We now proceed as in the proof of theorem
\ref{thm:PLA+Hw} i.e.\ we wish to apply lemma \ref{lem:main} for
some PLA function related to the null-series (\ref{eq:nullbeta}).
Fix some $N$ large and examine $e^{-i(\nu(N)-1)\beta}-1$ and $e^{-i\nu(N)\beta}-1$.
Since $\beta\in\left(\frac{2\pi}{3},\frac{4\pi}{3}\right)$, it is
not possible for both numbers to be small. We therefore examine two
cases: 
\begin{enumerate}
\item \label{enu:ni(N+1)}$|e^{-i\nu(N)\beta}-1|>(d(N)/N)\varepsilon(N/d(N))$.
\item \label{enu:ni(N)}$|e^{-i\nu(N)\beta}-1|\le(d(N)/N)\varepsilon(N/d(N))$.
This implies that $|e^{-i(\nu(N)-1)\beta}-1|>c$.
\end{enumerate}
Let us start with the first case (the other is similar but slightly
simpler). We shift the null-series (\ref{eq:nullbeta}) by $\nu(N)$
and get a new null-series whose positive part is the PLA expansion
of some PLA function, and whose negative part is its Fourier expansion.
Namely, define
\begin{align}
q(t): & =\frac{1}{d(N)|e^{-i\nu(N)\beta}-1|}\sum_{n>-\nu(N)}\gamma(n)(e^{i\beta n}-1)e^{i(n+\nu(N))t}\label{eq:qhat}\\
 & =1+\frac{1}{d(N)|e^{-i\nu(N)\beta}-1|}\sum_{n<-\nu(N)}\gamma(n)(e^{i\beta n}-1)e^{i(n+\nu(N))t}.\nonumber 
\end{align}
(the term ``1'' in the second line requires that $N$ be sufficiently
large because it requires that $\nu(N)\not\in Q$. But this follows
from our assumption (\ref{enu:ni(N+1)}) since if $\nu(N)=q(k)$then
$|e^{i\beta q(k)}-1|<C\varepsilon\varepsilon_{N}^{2}$ which contradicts
(\ref{enu:ni(N+1)}) for $N>C$). 

Now, by (\ref{eq:gammaneen}),
\[
\sum_{n<-\nu(N)}|\gamma(n)(e^{i\beta n}-1)|^{2}\le C\varepsilon\varepsilon_{N}^{2}.
\]
Hence 
\[
||q-1||_{2}\le\frac{C\varepsilon\varepsilon_{N}}{d(N)|e^{-i\nu(N)\beta}-1|}
\]
and since we assumed $|e^{-i\nu(N)\beta}-1|>(d(N)/N)\varepsilon(N/d(N))$
we get, 
\[
||q-1||_{2}\le\frac{CN\varepsilon\varepsilon_{N}}{d(N)^{2}\varepsilon(N/d(N))}.
\]
Recalling the definition of $\varepsilon\varepsilon_{n}$ (\ref{eq:defepseps}),
\[
||q-1||_{2}\le\frac{C}{N}\varepsilon\left(\frac{N^{2}}{d(N)^{2}}\varepsilon^{-1}\left(\frac{N}{d(N)}\right)\right)
\]
and if $N$ is sufficiently large the fraction is $<1$ and we can
apply lemma \ref{lem:main}. We get
\[
\left|\left\{ t:q^{*}(t)>\frac{N^{2}}{d(N)^{2}}\varepsilon^{-1}\left(\frac{N}{d(N)}\right)\right\} \right|>c_{1}.
\]
Recalling that the PLA expansion of $q$ is given by (\ref{eq:qhat})
we get that there is a set of measure $>c_{1}$ where 
\begin{multline}
\sup_{k}\left|\sum_{n=1-\nu(N)}^{k}\gamma(n)(e^{i\beta n}-1)e^{i(n+\nu(N))t}\right|>\\
>d(N)|e^{-i\nu(N)\beta}-1|\frac{N^{2}}{d(N)^{2}}\varepsilon^{-1}\left(\frac{N}{d(N)}\right)>N\label{eq:sumni(N)k}
\end{multline}
where the second inequality again uses our assumption (\ref{enu:ni(N+1)}).
We can replace in (\ref{eq:sumni(N)k}) the $e^{i(n+\nu(N))t}$ by
simply $e^{int}$ as it does not change the absolute value of the
expression. Finally to change the limit of the sum to 0 we note that
clearly 
\[
\sum_{n=1-\nu(N)}^{0}|c(n)(e^{i\beta n}-1)|^{2}\le C.
\]
This we may subtract from estimate (\ref{eq:sumni(N)k}) and get that
on a set of measure $>\frac{1}{2}c_{1}$, 
\begin{equation}
\sup_{k}\left|\sum_{n=0}^{k}\gamma(n)(e^{i\beta n}-1)e^{int}\right|>N-C\label{eq:endcaseI}
\end{equation}
and we are done with case (\ref{enu:ni(N+1)}).\medskip{}

\noindent \textbf{Step 7}. We are left with case (\ref{enu:ni(N)})
which is very similar, except that instead of shifting by $\nu(N)$
we shift by $\nu(N)-1$. The main reason to read this step is to see
why we needed to define $\varepsilon\varepsilon_{N}$ by iterating
$\varepsilon$ twice. As in case (\ref{enu:ni(N+1)}) for $N$ sufficiently
large we would have $\nu(N)-1\not\in Q$ so $\gamma(-\nu(N)+1)=d(N)$.
This gives, instead of (\ref{eq:qhat}),
\begin{align*}
q(t): & =\frac{1}{d(N)|e^{-i(\nu(N)-1)\beta}-1|}\sum_{n>1-\nu(N)}\gamma(n)(e^{i\beta n}-1)e^{i(n+\nu(N)-1)t}\\
 & =1+\frac{1}{d(N)|e^{-i(\nu(N)-1)\beta}-1|}\sum_{n<1-\nu(N)}\gamma(n)(e^{i\beta n}-1)e^{i(n+\nu(N)-1)t}.
\end{align*}
The argument that $||q-1||_{2}$ is small is similar. We have
\[
\sum_{n<1-\nu(N)}|\gamma(n)(e^{i\beta n}-1)|^{2}\le C\varepsilon\varepsilon_{N}^{2}+\frac{d(N)^{2}}{N^{2}}\varepsilon^{2}\left(\frac{N}{d(N)}\right)
\]
where the extra term is the one corresponding to $n=-\nu(N)$ and
is estimated by our assumption (\ref{enu:ni(N)}). The extra term
is the dominant one, so we may write
\[
\sum_{n<1-\nu(N)}|\gamma(n)(e^{i\beta n}-1)|^{2}\le C\frac{d(N)^{2}}{N^{2}}\varepsilon^{2}\left(\frac{N}{d(N)}\right).
\]
Since $|e^{-i(\nu(N)-1)\beta}-1|>c$ by our assumption, we get
\[
||q-1||_{2}\le\frac{1}{cd(N)}\cdot C\frac{d(N)}{N}\varepsilon\left(\frac{N}{d(N)}\right)
\]
so again for $N$ sufficiently large we may apply lemma \ref{lem:main}
and get
\[
\left|\left\{ t:q^{*}(t)>\frac{N}{d(N)}\right\} \right|>c_{1}
\]
the same argument as in the previous case then shows that on a set
of measure $>c_{1}$, 
\[
\sup_{k}\left|\sum_{n=2-\nu(N)}^{k}\gamma(n)(e^{i\beta n}-1)e^{int}\right|>cN
\]
and again on a set of measure $>\frac{1}{2}c_{1}$,
\begin{equation}
\sup_{k}\left|\sum_{n=0}^{k}\gamma(n)(e^{i\beta n}-1)e^{int}\right|>cN-C.\label{eq:endcaseII}
\end{equation}
As our conclusion (\ref{eq:endcaseI}) for case (\ref{enu:ni(N+1)})
is stronger, we in fact get that (\ref{eq:endcaseII}) holds regardless
of whether case (\ref{enu:ni(N+1)}) or case (\ref{enu:ni(N)}) held.

Since $N$ was arbitrary, we see that on a set of measure $>\frac{1}{2}c_{1}$
(the upper limit of the sets from (\ref{eq:endcaseII})), 
\[
\sup_{k}\left|\sum_{n=0}^{k}\gamma(n)(e^{i\beta n}-1)e^{int}\right|=\infty.
\]
In contradiction to our assumption after (\ref{eq:nullbeta}). Theorem
\ref{thm:PLA+UL} is thus proved.\qed

\end{document}

%% file: priv3.pstex_t
\begin{picture}(0,0)%
\includegraphics{priv3.pstex}%
\end{picture}%
\setlength{\unitlength}{3947sp}%
\begingroup\makeatletter\ifx\SetFigFont\undefined%
\gdef\SetFigFont#1#2#3#4#5{%
  \reset@font\fontsize{#1}{#2pt}%
  \fontfamily{#3}\fontseries{#4}\fontshape{#5}%
  \selectfont}%
\fi\endgroup%
\begin{picture}(1215,1174)(226,-711)
\put(1426,-169){\makebox(0,0)[lb]{\smash{{\SetFigFont{11}{13.2}{\rmdefault}{\mddefault}{\updefault}{\color[rgb]{0,0,0}$e^{it}$}%
}}}}
\put(526, 89){\makebox(0,0)[lb]{\smash{{\SetFigFont{11}{13.2}{\rmdefault}{\mddefault}{\updefault}{\color[rgb]{0,0,0}$Q_t$}%
}}}}
\end{picture}%

%% file: priv.pstex_t
\begin{picture}(0,0)%
\includegraphics{priv.pstex}%
\end{picture}%
\setlength{\unitlength}{3947sp}%
\begingroup\makeatletter\ifx\SetFigFont\undefined%
\gdef\SetFigFont#1#2#3#4#5{%
  \reset@font\fontsize{#1}{#2pt}%
  \fontfamily{#3}\fontseries{#4}\fontshape{#5}%
  \selectfont}%
\fi\endgroup%
\begin{picture}(2978,3056)(689,-4405)
\put(1426,-4336){\makebox(0,0)[lb]{\smash{{\SetFigFont{12}{14.4}{\rmdefault}{\mddefault}{\updefault}{\color[rgb]{0,0,0}$J$}%
}}}}
\put(1501,-3886){\makebox(0,0)[lb]{\smash{{\SetFigFont{12}{14.4}{\rmdefault}{\mddefault}{\updefault}{\color[rgb]{0,0,0}$J'$}%
}}}}
\end{picture}%